\documentclass[]{amsart}

\usepackage[foot]{amsaddr}

\author{Matthew D. Kvalheim}
\author{Philip Arathoon}

\address{Department of Mathematics and Statistics, University of Maryland, Baltimore County, Baltimore, MD, USA}

\address{Division of Mathematics, Analytics, Science, and Technology, Babson College, Wellesley, MA}

\email{kvalheim@umbc.edu, parathoon@babson.edu}

\makeatletter
\@namedef{subjclassname@2020}{%
  \textup{2020} Mathematics Subject Classification}
\makeatother   % MK: overriding since I only have up to 2010 classifications in my version of amsart
\title[Linearizability of flows by embeddings]{Linearizability of flows by embeddings}%{Linearizing Embeddability}
\usepackage{import} % Use import instead of input; for pdf_tex in figure folder

\usepackage{enumitem} % For itemized lists with modified labels

\usepackage{mathrsfs} % For script letter in Essen ref title
\usepackage[utf8]{inputenc}
\usepackage[T1]{fontenc}
\usepackage{lmodern}
\usepackage{amsmath}
\usepackage{amssymb}
\usepackage{amsthm}
\usepackage{stmaryrd} % for minnorm \llfloor, \rrfloor symbols
\usepackage{mathtools}
\usepackage{tikz-cd}
\usepackage{subfig}

\usepackage{thmtools}
\usepackage{thm-restate} % For restating theorem verbatim

\newcommand{\concept}[1]{\textbf{#1}}

\newcommand{\N}{\mathbb{N}}
\newcommand{\Z}{\mathbb{Z}}
\newcommand{\R}{\mathbb{R}}
\newcommand{\C}{\mathbb{C}}

\newcommand{\st}{X}
\newcommand{\att}{A}

\newcommand{\T}{T}

\newcommand{\id}{\textnormal{id}}

\theoremstyle{definition}
\newtheorem{Lem}{Lemma}
\newtheorem{Th}{Theorem}
\newtheorem{Co}{Corollary}

\newtheorem*{Quest-non}{Question}

\newtheorem*{Th-non}{Theorem}% Theorem without numbering

\newcommand{\thistheoremname}{}
\newtheorem*{genericthm}{\thistheoremname}
{\renewcommand{\thistheoremname}{Theorem~\ref{#1}$'$}%
	\begin{genericthm}}
	{\end{genericthm}}

\newtheorem*{Def*}{Definition}
\newtheorem{Ex}{Example}

\newtheorem{Rem}{Remark}

\usepackage{marvosym}
\usepackage[%hyperref
hypertexnames,%
citecolor=blue,%
colorlinks=true,%
linkcolor=red%
]{hyperref}
\usepackage[all]{hypcap}

\begin{document}
	%\sffamily
	
	\begin{abstract}
    We consider the problem of determining the class of continuous-time dynamical systems that can be globally linearized in the sense of admitting an embedding into a linear system on a 
    higher-dimensional Euclidean space.
    We solve this problem for dynamical systems on connected  state spaces that are either compact or contain at least one nonempty compact attractor, obtaining necessary and sufficient conditions for the existence of linearizing $C^k$ embeddings for $k\in \mathbb{N}_{\geq 0}\cup \{\infty\}$.
    Corollaries include (i) several checkable necessary conditions for global linearizability and (ii) extensions of the Hartman-Grobman and Floquet normal form theorems beyond the classical settings. 
    Our results open new perspectives on linearizability by establishing relationships to symmetry, topology, and invariant manifold theory.
	\end{abstract}
\maketitle

\tableofcontents

\section{Introduction}\label{sec:intro}
Consider a nonlinear system of ordinary differential equations
\begin{equation}\label{eq:ode}
\dot{x} = \frac{d}{dt}x = f(x),
\end{equation}
where $f$ is a vector field generating a continuous-time dynamical system or \emph{flow} $\Phi\colon \R\times M\to M$ on a manifold $M$, so that $t\mapsto \Phi^t(x)$ is the solution of \eqref{eq:ode} with initial condition $x = \Phi^0(x)$.
We say that a map $F\colon M\to \R^n$ is \concept{linearizing} if it is equivariant with respect to some linear flow on $\R^n$, i.e., if there is a matrix $B\in\R^{n\times n}$ such that
\begin{equation}\label{eq:lin-emb-intro}
	F\circ \Phi^t = e^{Bt}\circ F
\end{equation}
for all $t\in \R$.
In other words, $y(t)\coloneqq F(\Phi^t(x))$ solves the \emph{linear} system of equations
\begin{equation*}
	\dot{y}=By.
\end{equation*}
In this paper we study the existence of linearizing maps that are also $C^k$ embeddings, where $k\in \N_{\geq 0}\cup\{\infty\}$.\footnote{Standard definitions: a map is $C^0$ if it is continuous, and  is $C^k$ with $k\in \N_{\geq 1}\cup \{\infty\}$ if it is $C^0$ and has $C^0$ partial derivatives of all orders less than $k+1$ when expressed in local coordinates.
A \concept{$C^0$ embedding} (or \concept{topological embedding}) is a homeomorphism onto its image in the subspace topology \cite[p.~54]{lee2011topological}.
An \concept{immersion} is a $C^1$ map with an injective derivative at each point.
A \concept{$C^{k}$ embedding} with $k\in \N_{\geq 1}\cup \{\infty\}$ is a $C^k$ map that is both an immersion and a $C^0$ embedding \cite[p.~21]{hirsch1976difftop}.
Equivalently, a $C^{k\geq 1}$ map is a $C^k$ embedding if its image is a $C^k$ embedded submanifold and the map is a diffeomorphism from its domain to its image.
\label{foot:embedding}}
Note that $n\geq \dim M$ if $F\colon M\to \R^n$ is a $C^0$ embedding, i.e., embeddings always increase or preserve dimension.

\begin{Rem}\label{rem:complex-vs-real}
	One can define linearizing maps and $C^k$ embeddings $F\colon M\to \C^n$ in exactly the same way.
	But since $\R^n\subset \C^n$ and $\C^n$ is isomorphic to $\R^{2n}$ as a real vector space, a linearizing $C^k$ embedding $M\to \C^{n_1}$ exists for some $n_1$ if and only if a linearizing $C^k$ embedding $M\to \R^{n_2}$ exists for some $n_2$. 
	Thus, while we choose to emphasize $\R^n$-valued linearizing embeddings, the equivalent $\C^n$-valued viewpoint should be kept in mind.
	In particular, this viewpoint is useful for constructing linearizing embeddings given by concatenating eigenfunctions $M\to \C$ of the group of Koopman operators \cite{koopman1931hamiltonian,budisic2012appliedkoopmanism,mezic2020spectrum,mezic2021koopman,brunton2022modern}.
\end{Rem}

Conceptually, linearizing embeddings identify nonlinear systems with invariant subsets of linear systems.
We ask the fundamental

\begin{Quest-non}
	When do linearizing embeddings of a nonlinear system \eqref{eq:ode} exist?
\end{Quest-non}

We answer this question for the class of nonlinear systems on connected  state spaces that are either compact or contain at least one nonempty compact attractor, obtaining necessary and sufficient conditions.
(Our main results do not actually require a connected state space, but the preceding sentence is a convenient summary.)

These necessary and sufficient conditions will be illustrated  in \S \ref{sec:illustration} after discussing related work.
For now, we place them in the broader context of ``universality'' introduced by Tao to study the global existence problem for the Euler and Navier-Stokes equations \cite{tao2017universality,tao2018universalityi,tao2020universalityii}.
Slightly generalizing definitions of Tao \cite[p.~220]{tao2017universality} and Cardona and Presas \cite[pp.~3--4]{cardona2024hprinciple}, we say that a class $\mathcal{C}$ of flows is \concept{$C^k$-universal} with respect to another such class $\mathcal{D}$ if any flow in $\mathcal{D}$ admits an (equivariant) $C^k$ embedding into some flow in $\mathcal{C}$ (realizing the former as an invariant subset of the latter).
If we fix $$\mathcal{C}=\{\text{linear flows on finite-dimensional Euclidean spaces}\},$$ our existence question can be restated equivalently as follows: what is the largest class of flows  $\mathcal{D}$ with respect to which  $\mathcal{C}$ is $C^k$-universal?
We answer this question under the additional constraint that $\mathcal{D}\subset \mathcal{E}$, where $\mathcal{E}$ consists of flows on connected  spaces that are either compact or contain at least one nonempty compact attractor.

\subsection{Related work}\label{sec:related-work}
Poincar\'{e} initiated the study of \emph{locally}-defined, dimension-\emph{preserving} linearizing embeddings (homeomorphisms and diffeomorphisms) near an equilibrium point \cite{poincare1879proprietes, poincare1880oeuvres}, leading to classical linearization theorems including those of Poincar\'{e}-Siegel \cite{poincare1879proprietes,siegel1942iteration,siegel1952uber}, Sternberg \cite{sternberg1957local,sternberg1958structure}, Hartman-Grobman \cite{hartman1960lemma,grobman1959homeomorphism} and Hartman \cite{hartman1960local}. 
Extensions to locally-defined, dimension-preserving, \emph{partial} linearizations (in transverse directions) near general invariant manifolds were given in theorems including those of Pugh and Shub \cite{pugh1970linearization}, Robinson \cite{robinson1971differentiable}, Takens \cite{takens1971partiallyhyp}, Sell \cite{sell1983vector, sell1983linearization}, and Sakamoto \cite{sakamoto1994smooth}.
Such linearization results were extended to \emph{global} ones valid on the basin of attraction of a stable hyperbolic equilibrium or limit cycle by Lan and Mezi\'{c} \cite{lan2013linearization}
and, more generally, on the basin of attraction of a stable normally hyperbolic invariant manifold without boundary by Mezi\'{c} \cite{mezic_book} and with inflowing boundary by Eldering, Kvalheim, and Revzen \cite{eldering2018global}.

In this paper we are interested in \emph{globally}-defined, \textit{totally} (not partially) linearizing $C^k$ embeddings that are possibly dimension-\emph{increasing}.
Linearizing $C^k$ embeddings are conceptually natural objects of study, since they are precisely the maps $C^k$-identifying nonlinear systems with invariant subsets of linear systems.
However, previous literature has focused on related but different classes of maps. 

To our knowledge, Carleman was the first to study linearizing embeddings without the dimension-preserving assumption \cite{carleman1932application,kowalski1991nonlinear}.
Both Carleman and recent control theory literature consider $C^{k\geq 1}$ embeddings $F\colon M\to \R^n$ satisfying the extra requirements that $M=\R^m$ and $F\colon \R^m\to \R^n$ has a linear left inverse $\R^n\to \R^m$, while only satisfying a weaker ``linearizing'' property than we require.
Namely, the linear left inverse is merely required to send linear trajectories starting in the image of $F$ to nonlinear trajectories in $M$, with the image of $F$ not necessarily invariant under the linear flow.
Belabbas, Chen, and Ko dubbed such embeddings $F$ ``super-linearizations'' and defined ``strong super-linearizations'' to be super-linearizations that are also linearizing in our sense  \cite{belabbas2022canonical,belabbas2022visible,belabbas2023sufficient, ko2024minimum}.
One necessary and sufficient condition for super-linearizability of a $C^\infty$ flow $\Phi$ on $\R^m$ was obtained by Claude, Fliess, and Isidori \cite{claude1983immersion}.\footnote{They showed that a $C^\infty$ flow $\Phi$ on $\R^m$ is super-linearizable in this sense if and only if $\text{span}\{x\mapsto \frac{d^j}{dt^j}\Phi^t(x)|_{t=0}\colon j\in \N\}$ is a finite-dimensional linear subspace of the real vector space of $C^\infty$ maps $\R^m\to \R^m$ \cite{claude1983immersion} (see also Levine and Marino \cite[Lem.~1]{levine1986nonlinear}).}
Existence of a $C^k$ strong super-linearization for a flow $\Phi$ on $\R^m$ is also equivalent to its group of Koopman operators having a finite-dimensional invariant subspace of $C^k$ functions containing the state coordinate projections $(x_1,\ldots, x_m)\mapsto x_i$, a situation of interest in the recent ``applied Koopman operator'' literature \cite{brunton2016koopman}.
Super-linearizations are also related to ``polynomial flows'' \cite[p.~671]{essen1994locally}.

On the other hand, Liu, Ozay, and Sontag \cite{liu2023non,liu2025properties} recently obtained a necessary condition for existence of linearizing injective $C^0$ maps, called ``one-to-one linear immersions'' therein.
These are equivalent to $C ^0$ embeddings in situations considered in this paper (see Remark~\ref{rem:compare}).

Previously, Mezi\'{c} \cite{mezic2021koopman} introduced a necessary condition 
and a sufficient condition  
for existence of linearizing injective maps, called ``faithful linear representations'' therein.
These conditions involve properties of the Koopman operator acting on spaces of $L^2(\mu)$ functions, where $\mu$ is a suitable measure, and the components of the linearizing injective maps are assumed to be in  $L^2(\mu)$ but not necessarily continuous.
Thus, the sufficient condition \cite[Prop.~38]{mezic2021koopman} is not sufficient for linearizability by a $C^{0}$ embedding, but the necessary condition \cite[Cor.~33]{mezic2021koopman} is also necessary for linearizability by a $C^{0}$ embedding if the state space is compact (or is replaced by a compact invariant set) and $\mu$ is a finite Borel measure.

\subsection{Illustration of results}\label{sec:illustration}
We now illustrate our results on \eqref{eq:ode}, though some apply more generally.

One obstruction to the existence of a linearizing $C^0$ embedding, if $M$ is connected, occurs if \eqref{eq:ode} has a compact attractor whose basin is not all of $M$ (Corollary~\ref{co:motivate}).
Hence we are led to study the existence of linearizing embeddings defined on attractor basins or other subsets, i.e., linearizing embeddings of $(\st,\Phi)$ where $\st\subset M$ is a subset invariant under $\Phi$ (by abuse of notation we still write $\Phi$ for $\Phi|_{\R\times \st}$).
Such a linearizing embedding $F\colon \st\to \R^n$ satisfies \eqref{eq:lin-emb-intro} but is defined on $\st$ rather than $M$.

Our results are divided among four cases: linearization by a $C^{k\geq 1}$ embedding when $\st$ is an attractor basin or a compact invariant \emph{manifold}, and linearization by a $C^0$ embedding when $\st$ is an attractor basin or a compact invariant \emph{set}.

Our first set of results (\S\ref{subsec:cpct-smooth-case}) concern the case that $\st$ is a compact $C^{k\geq 1}$ invariant manifold.
(This includes the case $\st=M$ if $M$ is compact.)
Our main result for this case is Theorem~\ref{th:smooth-compact}, which asserts that $(\st,\Phi)$ is linearizable by a $C^k$ embedding if and only if $\Phi$ is a ``$1$-parameter subgroup'' (\S \ref{sec:results}) of a $C^k$ Lie group action of a torus on $\st$.
Examples include the familiar cases that $\st$ is an equilibrium, periodic orbit, or quasiperiodic invariant torus, but we give other examples in which $\st$ is a sphere, Klein bottle, or real projective plane with isolated equilibria (Examples~\ref{ex:basic-tori}, \ref{ex:sphere}, \ref{ex:klein}, \ref{ex:rp2}).
In fact, there seems to be confusion in the literature regarding linearizability by embeddings in the presence of isolated equilibria, so we also give related \emph{necessary} conditions for linearizability of $(\st,\Phi)$ in Proposition~\ref{prop:hopf} and its corollaries.
For example, if $\st$ is odd-dimensional and $\Phi$ has at least one isolated equilibrium, then $(\st,\Phi)$ is not linearizable by a $C^1$ embedding (Corollary~\ref{co:odd-isolated-no-embed}), and if $\dim \st=2$ and all equilibria of $\Phi$ are isolated, then $(\st,\Phi)$ cannot be linearized by a $C^1$ embedding unless $\st$ is diffeomorphic to either the $2$-torus, the $2$-sphere, the Klein bottle, or the real projective plane (Corollary~\ref{co:surfaces}).
We also give a separate \emph{sufficient} condition for linearizability of $(\st,\Phi)$ by a $C^k$ embedding in Proposition~\ref{prop:recog}.

Our second main result (Theorem~\ref{th:cont-compact}) concerns the case that $\st$ is a compact invariant set (not assumed to be a manifold).
It asserts that $(\st,\Phi)$ is linearizable by a $C^0$ embedding if and only if $\Phi$ is a $1$-parameter subgroup of a $C^0$ torus action that is not too pathological (has finitely many ``orbit types''). 
For example, all flows shown in Figure~\ref{fig:c0-linearizable} are linearizable by $C^0$ embeddings.

Our third main result (Theorem~\ref{th:cont}) concerns the case that $\st$ is the basin of attraction of an asymptotically stable compact invariant set $\att\subset M$.\footnote{Actually, $\att$ need only be globally asymptotically stable \emph{within} an invariant set $\st\subset M$, as long as $\st$ is ``locally closed'' in $M$ (see Remark~\ref{rem:top-assump-comment}).}
It asserts that $(\st,\Phi)$ is linearizable by a $C^0$ embedding if and only if $(\att,\Phi)$ is linearizable by a $C^0$ embedding (cf. Theorem~\ref{th:cont-compact}) and $\att$ has  $C^0$ ``asymptotic phase'' for $\Phi$.
Asymptotic phase as defined in \S \ref{subsec:cont-att-case} generalizes that associated with normally hyperbolic invariant manifolds \cite{fenichel1974asymp,hirsch1977invariant,bronstein1994smooth,eldering2018global}, which in turn  generalizes that associated with hyperbolic periodic orbits \cite{hale1980ordinary,guckenheimer1974isochrons,winfree2001geom}.
Theorem~\ref{th:cont} further asserts that linearizing $C^0$ embeddings (or, more generally, injective $C^0$ maps) are always \emph{proper} and hence limit to $\infty$ near the basin boundary $\partial \st$ (cf. \cite[Prop.~1]{haller2024data}), 
leading to the aforementioned statement (Corollary~\ref{co:motivate}) that $(M,\Phi)$ is not linearizable by a $C^0$ embedding if $M$ is connected and there is a non-global compact attractor.
Other corollaries of Theorem~\ref{th:cont} are extensions of the Hartman-Grobman and Floquet normal form theorems, at the cost of allowing ``extra'' linearizing coordinates (Corollaries~\ref{co:hartman-grobman}, \ref{co:floquet}).

Our fourth and final main result (Theorem~\ref{th:smooth}) again concerns the case that  $\st$ is the basin of attraction of an asymptotically stable compact invariant set $\att\subset M$, but instead concerns linearizability by a $C^{k\geq 1}$ embedding.
It asserts the following.
If $(\st,\Phi)$ is linearizable by a $C^k$ embedding, then $\att\subset \st$ must in fact be a $C^k$ embedded submanifold, 
$(\att,\Phi)$ must be  linearizable by a $C^k$ embedding (cf. Theorem~\ref{th:smooth-compact}), and $\att$ must have $C^k$ asymptotic phase.
Conversely, $(\st,\Phi)$ is linearizable by a $C^k$ embedding if $\att, \st, \Phi$ have the preceding properties, $\Phi$ is the flow of a $C^k$ vector field (cf. Remark~\ref{rem:uniq-int}), and $\Phi$ is ``transversely linearizable'' near $\att$ in a sense related to normal hyperbolicity (cf. Remarks~\ref{rem:nhim}, \ref{rem:about-cond-3}).

The remainder of the paper is organized as follows.
Our general results and examples are in \S\ref{sec:results}.
Proofs are in \S\ref{sec:proof}.
Closing remarks are  in \S\ref{sec:conclusion}.

\section{Results and examples}\label{sec:results}
This section contains our results and examples.
The $C^{k\geq 1}$ compact case is treated in \S\ref{subsec:cpct-smooth-case}, the $C^0$ compact case in \S\ref{subsec:cpct-continuous-case}, the $C^0$ attractor basin case in \S\ref{subsec:cont-att-case}, and the $C^{k\geq 1}$ attractor basin case in \S\ref{subsec:smooth-att-case}.
The main results (Theorems~\ref{th:smooth-compact}, \ref{th:cont-compact}, \ref{th:cont}, \ref{th:smooth}) are necessary and sufficient conditions for linearizing embeddability.
Other results include necessary conditions (Proposition~\ref{prop:hopf}, Corollaries~\ref{co:odd-isolated-no-embed}, \ref{co:euler}, \ref{co:surfaces}) and one sufficient condition (Proposition~\ref{prop:recog}) for linearizability in the $C^{k\geq 1}$ compact case, a necessary condition in the $C^0$ attractor basin case (Corollary~\ref{co:motivate}), and extensions of the Hartman-Grobman and Floquet normal form theorems (Corollaries~\ref{co:hartman-grobman}, \ref{co:floquet}).
The proofs of Theorems~\ref{th:smooth-compact}, \ref{th:cont-compact}, \ref{th:cont}, \ref{th:smooth} and Propositions~\ref{prop:hopf}, \ref{prop:recog} are deferred to \S\ref{sec:proof}.

A \concept{flow} on $\st$ is a group action $\Phi\colon \R\times \st\to \st$ of $\R$ on $\st$.
If $\Phi$ is a flow on $\st$, we say that a map $F\colon \st\to \R^n$ is \concept{linearizing} if there is $B\in \R^{n\times n}$ such that $F\circ \Phi^t = e^{Bt}\circ F$ for all $t\in \R$.
We say that $\Phi$ is a \concept{1-parameter subgroup} of a group action $\Theta\colon H\times \st\to \st$ of a torus $H= T^\ell = \R^\ell/\Z^\ell$  on $\st$ (a \concept{torus action}) if there is $\boldsymbol \omega\in \R^\ell$  such that $\Phi^t=\Theta^{(\boldsymbol \omega t \mod 1)}$, where  ``$\text{mod } 1$'' is applied entrywise.

\subsection{The smooth compact case}\label{subsec:cpct-smooth-case}
Here is the first main result.
\begin{Th}\label{th:smooth-compact}
Fix $k\in \N_{\geq 1}\cup \{\infty\}$.
Let $\Phi$ be a $C^k$ flow on a compact $C^k$ manifold $\st$.
Then $(X,\Phi)$ is linearizable by a $C^k$ embedding if and only if $\Phi$ is  a $1$-parameter subgroup of a $C^k$ torus action.
\end{Th}

\begin{Rem}
With the aid of Theorem~\ref{th:smooth-compact}, the class of linearizable systems $(\st,\Phi)$ as in its statement can be seen to be a subclass of the class of ``non-Hamiltonian integrable systems'' from the literature \cite[Def.~2.6]{zung2006torus}.
\end{Rem}

The following examples illustrate Theorem~\ref{th:smooth-compact}.
\begin{Ex}\label{ex:zero-vf}
Let $\st$ be any compact $C^\infty$ manifold and $\Phi$ be the flow of the zero vector field.
By the Whitney embedding theorem, $(\st,\Phi)$ admits a $C^\infty$ embedding into the linear flow of the zero vector field on some Euclidean space.
Thus, $(\st,\Phi)$ is linearizable by a $C^\infty$ embedding.
Note that $\Phi$ coincides with any 1-parameter subgroup of the trivial action of any torus.
\end{Ex}

The remaining examples also serve to motivate Proposition~\ref{prop:hopf} and its corollaries.

\begin{Ex}\label{ex:basic-tori}
Suppose $(\st,\Phi)$ is either (i) a single equilibrium point, (ii) a single periodic orbit viewed as a constant-speed rotation on the circle $\st = S^1$, or (iii) a single quasiperiodic torus viewed as a product of constant-speed rotations on the ($n$-)torus $\st = T^n = S^1\times\cdots\times S^1$, a product of circles.
It is clear (and well-known) that all of these cases are linearizable by $C^\infty$ embeddings.
Let us compare this fact with Theorem~\ref{th:smooth-compact}.
In all cases, $(\st,\Phi)$ is  a $1$-parameter subgroup of a $C^\infty$ torus action on $\st$: the actions $\Theta$ can be respectively taken to be  (i) the trivial action of $T^1$ on the equilibrium, (ii) the action of $S^1=T^1$ on itself, or (iii) the action of $T^n$ on itself. 
\end{Ex}

Despite contrary claims in the literature,
the following example shows that it is possible for $(\st,\Phi)$ in Theorem~\ref{th:smooth-compact} to be linearizable by a $C^\infty$ embedding even if $(\st,\Phi)$ has multiple isolated equilibrium points.

\begin{Ex}\label{ex:sphere}
Let $\st = S^2\subset \C\times \R\approx \R^3$ be the unit $2$-sphere centered at the origin.
Let $\Theta$ be the $C^\infty$ action of the circle $S^1\approx \R/\Z$ on $\C\times \R$ defined by $\Theta^h(z,s)\coloneqq (e^{2\pi i h}z, s)$.
This action leaves $S^2$ invariant because $\Theta$ simply rotates the sphere's latitudinal circles.
Thus, Theorem~\ref{th:smooth-compact} implies that $(\st,\Phi)$ with $\Phi^t\coloneqq \Theta^{t\mod 1}$ is linearizable by a $C^\infty$ embedding.
This conclusion can also be seen directly since $\Phi$ is simply the restriction to $\st\subset \C\times \R$ of the linear flow $(t,z,y)\mapsto (e^{2\pi i t}z,y)$ on $\C\times \R \approx \R^3$.
In anticipation of Proposition~\ref{prop:hopf}, note that $(\st,\Phi)$ has two equilibria, each with Hopf index \cite[pp.~133--134]{guillemin1974differential} equal to $1$.
\end{Ex}

Tori and spheres are orientable.
Here are two examples with nonorientable $\st$.

\begin{Ex}\label{ex:klein}
Let $\st$ be the Klein bottle viewed as the quotient of $T^2 = \R^2/\Z^2$ by the group action of $\Z_2$ generated by $(x,y)\mapsto (x+\frac{1}{2},-y) \mod 1$.
The $C^\infty$ action of $S^1=\R/\Z$ on $\R^2/\Z^2$ by translations of the first factor commutes with the $\Z_2$ action and hence descends to a $C^\infty$ action $\Theta$ of $S^1$ on $\st$.
Thus, by Theorem~\ref{th:smooth-compact}, $(\st,\Phi)$ with $\Phi^t\coloneqq \Theta^{t\mod 1}$ is linearizable by a $C^\infty$ embedding.
In anticipation of Proposition~\ref{prop:hopf}, note that $(\st,\Phi)$ has no equilibria.
\end{Ex}

\begin{Ex}\label{ex:rp2}
Let $\st = \R P^2$ be the real projective plane viewed as the quotient of the $2$-sphere $S^2\subset \C\times \R\approx \R^3$ by the action of $\Z_2$ generated by the antipodal map.
The $C^\infty$ $S^1$ action from Example~\ref{ex:sphere} commutes with the $\Z_2$ action, so it descends to a $C^\infty$ $S^1$ action $\Theta$ on $\st$ having one fixed point.
Thus, by Theorem~\ref{th:smooth-compact}, $(\st,\Phi)$ with $\Phi^t\coloneqq \Theta^{t\mod 1}$ is linearizable by a $C^\infty$ embedding.
In anticipation of Proposition~\ref{prop:hopf}, note that $(\st,\Phi)$ has one equilibrium with Hopf index equal to $1$.
\end{Ex}

Example \ref{ex:zero-vf} showed that any compact $C^\infty$ manifold $\st$ admits \emph{some} flow $\Phi$ that can be linearized by a $C^\infty$ embedding.
However, Examples~\ref{ex:basic-tori},  \ref{ex:sphere}, \ref{ex:klein}, \ref{ex:rp2} motivate the following proposition and corollaries restricting linearizability of flows having isolated equilibria, i.e., flows generated by vector fields with only isolated zeros.
(The proof of Proposition~\ref{prop:hopf} uses Theorem~\ref{th:smooth-compact}.)

\begin{restatable}[]{Prop}{PropHopf}\label{prop:hopf}
Let $\Phi$ be a $C^1$ flow on a connected compact $C^1$ manifold $\st$.
Assume that $(\st,\Phi)$ can be linearized by a $C^1$ embedding and that $\Phi$ has at least one isolated equilibrium.
Then $\st$ is even-dimensional, and the Hopf index of the vector field generating $\Phi$ at any isolated equilibrium is equal to $1$.
\end{restatable}

\begin{Co}\label{co:odd-isolated-no-embed}
If $\Phi$ is a $C^1$ flow on an odd-dimensional connected compact $C^1$ manifold $\st$ with at least one isolated equilibrium, then $(\st,\Phi)$ is not linearizable by a $C^1$ embedding.
\end{Co}

Proposition~\ref{prop:hopf} and the Poincar\'e-Hopf theorem \cite[p.~35]{milnor1965topology} imply the following corollary, which is consistent with Examples~\ref{ex:basic-tori}, \ref{ex:sphere}, \ref{ex:klein}, \ref{ex:rp2}.

\begin{Rem}\label{rem:c1-vs-cinf}
The standard Poincar\'{e}-Hopf theorem  and other results for $C^\infty$ manifolds also apply to $C^1$ manifolds since any $C^1$ manifold has a compatible $C^\infty$ structure \cite[Thm~2.2.9]{hirsch1976difftop}.	
\end{Rem}

\begin{Co}\label{co:euler}
Let $\st$ be a compact $C^1$ manifold.
Assume that there exists a $C^1$ flow $\Phi$ with only finitely many equilibria such that $(\st,\Phi)$ is linearizable by a $C^1$ embedding.
Then the Euler characteristic $\chi(\st) = \#\{\text{equilibria}\}\geq 0$.
\end{Co}

The following corollary follows directly from Corollary~\ref{co:euler}, Examples~\ref{ex:basic-tori}, \ref{ex:sphere}, \ref{ex:klein}, \ref{ex:rp2}, and the classification of surfaces \cite[Thm~9.3.5]{hirsch1976difftop}.
\begin{Co}\label{co:surfaces}
Let $\st$ be a $2$-dimensional connected compact $C^\infty$ manifold. 
The following are equivalent:
\begin{itemize}
	\item There exists a $C^1$ flow $\Phi$ on $\st$ with only finitely many equilibria such that $(\st,\Phi)$ is linearizable by a $C^1$ embedding.
	\item There exists a $C^\infty$ flow $\Phi$ on $\st$ with only finitely many equilibria such that $(\st,\Phi)$ is linearizable by a $C^\infty$ embedding.
	\item $\st$ is diffeomorphic to either the $2$-torus, the $2$-sphere, the Klein bottle, or the real projective plane.
\end{itemize}
\end{Co}

Theorem~\ref{th:smooth-compact} gives a necessary \emph{and} sufficient condition for linearizability by a $C^{k\geq 1}$ embedding.
On the other hand,  Proposition~\ref{prop:hopf} and Corollaries~\ref{co:euler}, \ref{co:surfaces} give \emph{necessary} conditions for the same, while the following proposition gives a \emph{sufficient} condition.
Proposition~\ref{prop:recog} is a ``rigidity'' statement about Koopman eigenfunctions, and its proof is independent of our other results.
(Since $\st$ is compact, assuming the existence of $F$ below is readily seen to be equivalent to assuming the existence of $n$  nowhere-zero $C^k$ Koopman eigenfunctions $\st\to \C$ whose  eigenvalues are rationally independent.) 

\begin{restatable}[]{Prop}{PropRecog}\label{prop:recog}
Fix $k\in \N_{\geq 1}\cup \{\infty\}$.
Let $\Phi$ be a $C^k$ flow on a connected compact $n$-dimensional $C^k$ manifold $\st$.
Assume there is a $C^k$ map $F\colon \st \to T^n$ to the $n$-torus $T^n = \R^n/\Z^n$ and a vector $\boldsymbol \omega\in \R^n$ with rationally independent components such that $F\circ \Phi^t(x) = \boldsymbol \omega t + F(x) \mod 1$ for all $x\in \st$, $t\in \R$.
Then there is a $C^k$ diffeomorphic identification $\st\approx T^n$ with respect to which $F$ is induced by an invertible matrix with integer entries and $\Phi$ is the flow of a constant vector field $\tilde{\boldsymbol \omega}\in \R^n$ on $T^n$ with rationally independent components.
In particular, $(\st,\Phi)$ is linearizable by a $C^k$ embedding.
\end{restatable}

\subsection{The continuous compact case}
\label{subsec:cpct-continuous-case}

Theorem~\ref{th:smooth-compact} can be applied to flows on compact manifolds, or to the restriction to a compact invariant manifold of a flow on an ambient manifold.
However, compact invariant sets of smooth flows on manifolds are not generally manifolds.
Thus, it is useful to have the following Theorem~\ref{th:cont-compact}, a ``continuous'' version of the ``smooth'' Theorem~\ref{th:smooth-compact}, for application to general compact invariant sets (and for application in the abstract setting of topological dynamics).

A $C^0$ action of a torus $\Theta\colon T\times \st\to \st$ on a topological space $\st$ has \concept{finitely many orbit types} if there are only finitely many subgroups $H\subset T$ with the property that $H=\{g\in T\colon \Theta^g(x)=x\}$ for some $x\in \st$ \cite[Def.~1.2.1]{palais1960classification}.

\begin{Th}\label{th:cont-compact}
Let $\st$ be a compact topological space homeomorphic to a subspace of a  manifold, and $\Phi$ be a $C^0$ flow on $\st$.
Then $(X,\Phi)$ is linearizable by a $C^0$ embedding if and only if $\Phi$ is  a $1$-parameter subgroup of a $C^0$ torus action with finitely many orbit types.
\end{Th}

For example, it is immediate from Theorem~\ref{th:cont-compact}  that each flow in Figure~\ref{fig:c0-linearizable} is linearizable by a $C^0$ embedding.

\begin{figure}
	\centering
	%\def\svgwidth{1.0\linewidth}
	%\import{figs/}{all-3.pdf_tex}
\includegraphics[width=1.0\linewidth]{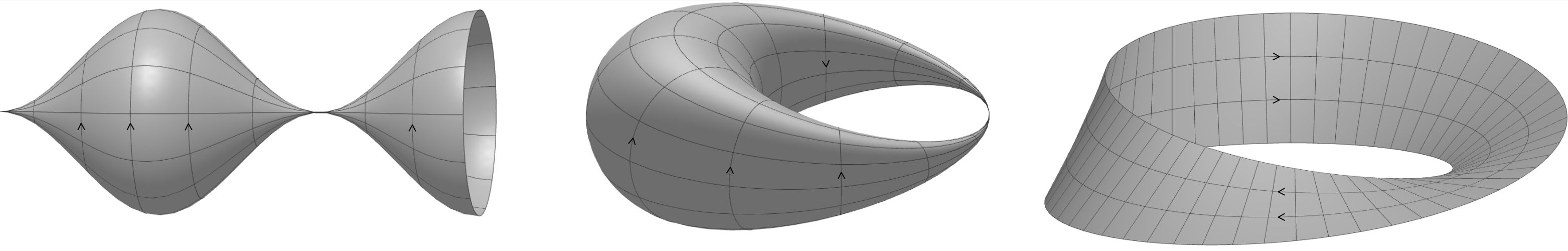}	
	\caption{Examples of flows that are linearizable by $C^0$ embeddings.
		This follows from Theorem~\ref{th:cont-compact} since each state space may be viewed as a compact subset of $\R^3$, and each flow is a 1-parameter subgroup of a $C^0$ torus (circle) action with finitely many orbit types. 
		The rightmost example is actually linearizable by a $C^\infty$ embedding, as a direct construction shows.
	}\label{fig:c0-linearizable}
\end{figure}

\begin{Rem}\label{rem:compare}
	To facilitate  comparison of Theorem~\ref{th:cont-compact} with a result of Liu, Ozay, and Sontag \cite[Cor.~3]{liu2023non,liu2025properties} (see Remark~\ref{rem:los-complement}), note that Theorem~\ref{th:cont-compact} remains true if ``$C^0$ embedding'' is replaced by ``injective $C^0$ map''.
	(This is because any injective $C^0$ map from a compact space to a Hausdorff space is a $C^0$ embedding.)
	The same is true of Theorem~\ref{th:cont} (by its final sentence). 
\end{Rem}

\begin{Rem}\label{rem:orbit-types}
	Theorem~\ref{th:cont-compact} explicitly restricts attention to $C^0$ torus actions with finitely many orbit types. This is not necessary in Theorem~\ref{th:smooth-compact}, since $C^1$ torus actions on compact manifolds always have finitely many orbit types \cite[Prop.~2.7.1]{duistermaat2000lie}.
\end{Rem}

\begin{Rem}\label{rem:top-assump-comment-compact}
The assumption that $\st$ is homeomorphic to a subspace of a manifold has nothing to do with $\Phi$.
By the Whitney embedding theorem, it is equivalent to the assumption that $\st$ admits a $C^0$ embedding into some $\R^n$, which is clearly necessary for the existence of such a $C^0$ embedding that is also linearizing.
The same assumption can also be reformulated  intrinsically: a compact topological space $\st$ is homeomorphic to a subspace of a manifold if and only if  $\st$ is metrizable and has finite Lebesgue covering dimension \cite[Thm~50.1, p.~316]{munkres2000topology}.
\end{Rem}

\subsection{The continuous attractor basin case}\label{subsec:cont-att-case}
Theorems~\ref{th:cont}, \ref{th:smooth}  make use of the following notion.
If $\att\subset \st$ is a globally asymptotically stable compact invariant set for a flow $\Phi$ on a topological space $\st$, $P\colon\st\to \att$ is an \concept{asymptotic phase map} for $\att$ if $P$ is a retraction ($P|_\att = \id_\att$) and $P\circ \Phi^t = \Phi^t\circ P$ for all $t\in \R$.
We say that $\att$ has \concept{$C^0$ asymptotic phase} if a $C^0$ asymptotic phase map for $\att$ exists.

\begin{Rem}\label{rem:asymp-in-phase}
If $\att$ has $C^0$ asymptotic phase $P$ and $\st$ is a metric space, then each $x\in \st$ is ``asymptotically in phase with'' $P(x)$ since $P|_\att = \id_\att$ and $\Phi^t(x)\to \att$, so $$\lim_{t\to\infty}\text{dist}(\Phi^t(x),\Phi^t(P(x)))= \lim_{t\to\infty}\text{dist}(\Phi^t(x),P(\Phi^t(x))) = 0.$$
\end{Rem}

Recall that a map $F\colon X\to Y$ between topological spaces is \concept{proper} if $F^{-1}(K)$ is compact for every compact subset $K \subset Y$ \cite[p.~118]{lee2011topological}.

\begin{restatable}[]{Th}{ThmCont}\label{th:cont}
Let $\st$ be a locally compact topological space homeomorphic to a subspace of a manifold.
Let $\Phi$ be a $C^0$ flow on $\st$ with a globally asymptotically stable compact invariant set $\att\subset \st$.
Then $(\st,\Phi)$ is linearizable by a $C^0$ embedding if and only if:
\begin{enumerate}
\item\label{item:cont-1} $\att$ has $C^0$ asymptotic phase, and 
\item\label{item:cont-2} the restricted flow $\Phi|_{\R\times \att}$ is  a $1$-parameter subgroup of a $C^0$ torus action with finitely many orbit types.
\end{enumerate}
Moreover, any linearizing injective $C^0$ map is proper, hence also a $C^0$ embedding.
\end{restatable}

Specializing Theorem~\ref{th:cont} to $\att = \{x_*\}$ yields the following extension of the (globalized \cite{lan2013linearization,mezic_book,eldering2018global,kvalheim2021existence}) classical Hartman-Grobman theorem \cite{hartman1960lemma,grobman1959homeomorphism}.
The latter guarantees linearizability by a dimension-\emph{preserving} $C^0$ embedding when $\st$ is a $C^\infty$ manifold, $\Phi$ is the flow of a $C^1$ vector field, and $x_*$ is hyperbolic.

\begin{Co}\label{co:hartman-grobman}
	Let $\st$ be a locally compact topological space homeomorphic to a subspace of a manifold.
	Let $\Phi$ be a $C^0$ flow on $\st$ with a globally asymptotically stable equilibrium point $x_*\in \st$.
	Then $(\st,\Phi)$ is linearizable by a $C^0$ embedding.
\end{Co}

Similarly, specializing Theorem~\ref{th:cont} to the image $\att$ of a limit cycle yields the following extension of the (globalized \cite{lan2013linearization,mezic_book,eldering2018global,kvalheim2021existence}) classical Floquet normal form theorem (\cite[Sec.~26]{abraham1967transversal}, \cite[Sec.~I.3]{bronstein1994smooth}, \cite[Sec.~4.3]{abraham2001manifolds}), a nonlinear generalization of the classical Floquet theory of linear time-periodic systems \cite[Sec.~III.7]{hale1980ordinary}.
The standard Floquet normal form theorem guarantees linearizability by a $C^0$ embedding when $\st$ is a $C^\infty$ manifold, $\Phi$ is the flow of a $C^1$ vector field, and $\att$ is normally hyperbolic.

\begin{Co}\label{co:floquet}
	Let $\st$ be a locally compact topological space homeomorphic to a subspace of a manifold.
	Let $\Phi$ be a $C^0$ flow on $\st$ with a globally asymptotically stable periodic orbit with image $\att$.
	Then $(\st,\Phi)$ is linearizable by a $C^0$ embedding if and only if $\att$ has $C^0$ asymptotic phase.
\end{Co}

\begin{Rem}\label{rem:top-assump-comment}
	Similarly to Remark~\ref{rem:top-assump-comment-compact}, 	the assumption that $\st$ is locally compact and homeomorphic to a subspace of a manifold has nothing to do with $\Phi$.
	It is equivalent to the assumption that $\st$ is homeomorphic to a \concept{locally closed} subspace $S$ of some manifold $M$, meaning that $S$ is the intersection of an open subset and a closed subset of $M$  (and is equipped with the subspace topology). 
	For example, any open or closed subset of $M$ is itself locally closed in $M$ since it is the intersection of $M$ with itself, and $M$ is both a closed and an open subset of itself.
	Thus, in particular, Theorem~\ref{th:cont} applies to the restriction $\Phi$ of a flow on a manifold $M$ to the basin of attraction $\st\subset M$ of an asymptotically stable compact invariant set $\att\subset M$, because such a basin $\st$ is always open in $M$.
	Finally, the same assumption can also be reformulated intrinsically: a locally compact topological space $\st$ is homeomorphic to a subspace of a manifold if and only if $\st$ is metrizable, separable, and has finite Lebesgue covering dimension (\cite[Thm~50.1, p.~316]{munkres2000topology}, \cite[Cor.~3.3.10]{engelking1989general}). 
\end{Rem}

\begin{Ex}\label{ex:stabilize-earlier}
	Let $\Phi_0$ be any $C^0$ flow on a compact space $\att_0$ that is linearizable by a $C^0$ embedding $F_0\colon A_0\to \R^{n_0}$, such as any of the flows from Examples~\ref{ex:zero-vf}, \ref{ex:basic-tori}, \ref{ex:sphere}, \ref{ex:klein}, \ref{ex:rp2}.
	Fix $\ell\in \N$ and let $\Phi_1$ be any $C^0$ flow on $\R^\ell$ for which the origin is globally asymptotically stable.
	Then $\att\coloneqq A_0\times \{0\}$ is a globally asymptotically stable compact invariant set for the flow $\Phi$ on $\st\coloneqq A_0\times \R^\ell$ defined by $\Phi^t = \Phi_0^t \times \Phi_1^t$ with $C^0$ asymptotic phase $(a,y)\mapsto a$.
	Thus, Theorem~\ref{th:cont} implies that $(\st,\Phi)$ is linearizable by a $C^0$ embedding.
	On the other hand, we can also use Corollary~\ref{co:hartman-grobman} to obtain a linearizing $C^0$ embedding of a special form.
	Namely, Corollary~\ref{co:hartman-grobman} yields a linearizing $C^0$ embedding $F_1\colon \R^\ell \to \R^{n_1}$ for $(\R^\ell, \Phi_1)$, and  $F_0\times F_1\colon A_0\times \R^\ell\to \R^{n_0}\times \R^{n_1}$ is readily checked to be a linearizing $C^0$ embedding for $(\st, \Phi)$.
\end{Ex}

\begin{Ex}\label{ex:nonhyp-po-nophase}
	Consider the $C^\infty$ flow $\Phi$ on $\st\coloneqq \R^2\setminus \{0\}$ generated by the system of ordinary differential equations, in polar coordinates $(r,\theta)$, 
	$$\dot r=-(r-1)^3, \qquad \dot\theta = r.$$
	Solving this system in closed form yields
	$$\Phi^t(1,\theta)=(1,\theta+t \mod 2\pi)$$
	and 
	$$\Phi^t(r,\theta)=\left(1+\frac{1}{\sqrt{2t+(r-1)^{-2}}},\theta-\frac{1}{r-1}+t+\sqrt{2t+(r-1)^{-2}} \mod 2\pi\right)$$
	when $r\neq 1$.
	Since $\sqrt{2t+(r-1)^{-2}}\to\infty$ as $t\to\infty$, the unit circle $\att = \{r=1\}$ is globally asymptotically stable, but each trajectory in $\R^2\setminus \att$ is \emph{not} asymptotically in phase with any trajectory in $\att$.
	Thus, Theorem~\ref{th:cont} implies that $(\st,\Phi)$ is \emph{not} linearizable by an injective $C^0$ map, even though Example~\ref{ex:basic-tori} showed that the restriction $(\att,\Phi|_{\R\times \att})$ is linearizable by a $C^0$ embedding.
\end{Ex}

Applying the ``proper'' statement of Theorem~\ref{th:cont} to a hypothetical linearizing injective $C^0$ map restricted to a basin of attraction yields the next corollary, which motivated our attention to basins of attractions (or, equivalently, global attractors) in Theorems~\ref{th:cont}, \ref{th:smooth}.

\begin{Co}\label{co:motivate}
Let $\st$ be a connected and locally compact topological space homeomorphic to a subspace of a manifold.
Let  $\Phi$ be a $C^0$ flow on $\st$ with a nonempty asymptotically stable compact invariant set $\att\subset \st$ whose basin of attraction is not equal to $\st$.
Then $(\st, \Phi)$ is not linearizable by an injective $C^0$ map.	
\end{Co}

\begin{Rem}\label{rem:los-complement}
	Corollary~\ref{co:motivate} is complementary to a result of Liu, Ozay, and Sontag \cite[Cor.~3]{liu2023non,liu2025properties}.
	Under some additional assumptions their result asserts that, if all forward $\Phi$-orbits are precompact and the collection of all omega-limit sets of points $x\in \st$ is countable and contains more than one element, then $(\st,\Phi)$ is not linearizable by an injective $C^0$ map.
	Similarly, for $\st$ connected  Corollary~\ref{co:motivate} implies that, if  $\Phi$ has more than one nonempty compact asymptotically stable invariant set, then $(\st,\Phi)$ is not linearizable by an injective $C^0$ map.
	Corollary~\ref{co:motivate} and \cite[Cor.~3]{liu2023non,liu2025properties} are independent since not every omega-limit set is an asymptotically stable set, and conversely.
\end{Rem}

\subsection{The smooth attractor basin case}\label{subsec:smooth-att-case}
As in \S\ref{subsec:cont-att-case}, if $\att\subset \st$ is a globally asymptotically stable compact $C^{k}$ embedded invariant manifold for a flow $\Phi$ on a $C^k$ manifold $\st$, we say that $\att$ has \concept{$C^k$ asymptotic phase} if a $C^k$ asymptotic phase map $P\colon \st\to \att$ for $\att$ exists.

For the following theorem, let $\pi\colon \T\st\to \st$ be the tangent bundle and $\att\subset \st$ be a subset.
Then $\T_{\att}\st\coloneqq \pi^{-1}(\att)$ denotes the tangent bundle of $\st$ over $\att$, and if $U\supset \att$ is open and $G\colon U\to \R^\ell$ is $C^1$, then $\T_\att G\colon\T_\att \st\to \T \R^\ell$ denotes the restriction $\T G|_{\T_\att \st}$ of the tangent map of $G$ to $\T_\att \st$.

\begin{restatable}[]{Th}{ThmSmooth}\label{th:smooth}
Fix $k\in \N_{\geq 1}\cup\{\infty\}$.
Let $\Phi$ be a $C^k$ flow of a uniquely integrable vector field
on a $C^k$ manifold $\st$ with a globally asymptotically stable compact invariant set $\att\subset \st$.
Then $(\st,\Phi)$ is linearizable by a $C^k$ embedding if and only if:
\begin{enumerate}
\item\label{item:smooth-1} $\att$ is a $C^k$ embedded submanifold of $\st$ with $C^k$ asymptotic phase;

\item\label{item:smooth-2}  the restricted flow $\Phi|_{\R\times \att}$ is  a $1$-parameter subgroup of a $C^k$ torus action; and 

\item\label{item:smooth-3}there is $\ell\in \N$,  a matrix $B\in \R^{\ell\times \ell}$ with all eigenvalues having negative real part, an open set $U\supset \att$, and a $C^k$ map $G\colon U\to \R^\ell$ satisfying $\ker(\T_\att G)= \T \att$ and $G(\Phi^t(x)) = e^{Bt} G(x)$ for all $x\in U$ and $t\in \R$ such that $\Phi^t(x)\in U$.
\end{enumerate}
Moreover, any linearizing $C^k$ embedding is a proper map.
\end{restatable}

\begin{Rem}\label{rem:uniq-int}
Here \concept{uniquely integrable} means that the vector field $x\mapsto \frac{d}{dt}\Phi^t(x)|_{t=0}$ has unique maximal trajectories.
By the Picard-Lindel\"{o}f theorem this holds if $k\geq 2$ or, more generally, if the vector field generating $\Phi$ is locally Lipschitz.
\end{Rem}

\begin{Ex}\label{ex:nonresonance}
Let $(\st,\Phi)$ be the basin of attraction of a stable hyperbolic limit cycle for a $C^{\infty}$ flow.
In this case, it is known that Conditions~\ref{item:cont-1}, \ref{item:cont-2} are always satisfied with $k=\infty$.
Thus, such a flow $(\st,\Phi)$ is linearizable by a $C^\infty$ embedding if and only if Condition~\ref{item:smooth-3} is also satisfied with $k=\infty$.
This is the case, e.g., under the typically-satisfied 
condition that the Floquet multipliers associated with the linearized flow are nonresonant \cite[Prop.~3]{kvalheim2021existence}.
Note that any linearizing $C^0$ embedding must be dimension-\emph{increasing} for this example.
(Remark~\ref{rem:about-cond-3} discusses Condition~\ref{item:smooth-3} further.)
\end{Ex}

\begin{Rem}\label{rem:dimension-varying}
Note that $\st$ is connected if and only if $\att$ is connected  \cite[Thm~6.3]{gobbino2001topological}.
If $\st$ is not connected, then Condition~\ref{item:smooth-1} should be understood in the sense that the (finitely many) connected components of $\att$ are compact $C^k$ embedded submanifolds of possibly different dimensions. 
\end{Rem}

\begin{Rem}\label{rem:nhim}
If $(\st,\Phi)$ is linearizable by a $C^{k\geq 1}$ embedding, then $\att$ must be a normally hyperbolic invariant manifold (NHIM).
More precisely, $\att$ is an eventually relatively $\infty$-NHIM (\cite[p.~4]{hirsch1977invariant}, \cite[p.~4207]{eldering2018global}) with respect to any Riemannian metric on $\st$.
This follows readily from the proof of Theorem~\ref{th:smooth}, which shows that $(X,\Phi)$ is linearizable by a proper $C^k$ embedding $F\colon\st\to \R^n$ such that $F(\att)$ is contained in the real invariant subspace for the matrix generating the linear flow corresponding to eigenspaces of purely imaginary eigenvalues.
The same reasoning implies that $\att$ satisfies ``center bunching'' conditions \cite[Eq.~(11)]{eldering2018global} of all orders, consistent with the existence of $C^k$ asymptotic phase implied by such center bunching conditions \cite[Cor.~2]{eldering2018global}. 
\end{Rem}

\begin{Rem}\label{rem:about-cond-3}
Theorem~\ref{th:smooth} is less satisfying than Theorems~\ref{th:smooth-compact}, \ref{th:cont-compact}, \ref{th:cont} since, roughly speaking, verifying Condition~\ref{item:smooth-3}  seems roughly half as difficult as directly verifying that $(\st,\Phi)$ is linearizable by a $C^k$ embedding. 
However, if $\att$ is a compact $C^k$ embedded NHIM, three conditions that together imply Condition~\ref{item:smooth-3} are the following, where the tangent flow $\T \Phi\colon \R\times \T \st\to \T \st$ is defined (with some abuse of notation) by $(\T\Phi)^t = \T \Phi^t$: 
\begin{enumerate}[label=(\alph*)]
\item\label{item:nhim-es-ck-triv} The ``stable vector bundle'' (\cite[p.~1]{hirsch1977invariant}, \cite[p.~4207]{eldering2018global}) $E^s\subset \T_\att \st$ of $\att$  is $C^k$ and globally trivializable,

\item\label{item:nhim-linearizable} $\Phi$ is $C^{k+1}$ and locally $C^k$ conjugate to $\T\Phi|_{\R\times E^s}$, and

\item\label{item:nhim-es-red}  $\T\Phi|_{\R \times E^s}$ is ``$C^k$ reducible'' (conjugate via a $C^k$ vector bundle automorphism covering $\id_\att$) to the product of $\Phi|_{\R\times \att}$ with a constant (with respect to some global trivialization of $E^s$) linear flow.

\end{enumerate}
Regarding \ref{item:nhim-es-ck-triv}, note that $E^s$ is globally trivializable if and only if $\att$ is a level set of some $C^k$ submersion, and a sufficient condition implying $E^s$ is $C^k$ is that $\att$ satisfies ``$k$-center bunching'' conditions \cite[Eq.~11]{eldering2018global}.
Assuming $\Phi$ is $C^{k+1}$, sufficient conditions for \ref{item:nhim-linearizable} arise from $C^k$ local linearization theorems for NHIMs \cite{sternberg1957local,takens1971partiallyhyp,robinson1971differentiable,sell1983linearization,sell1983vector,sakamoto1994smooth,bronstein1994smooth} combined with a globalization technique \cite[sec.~4.6]{mezic_book} (see also \cite[Thm~2]{eldering2018global}).
However, even in the case of Sternberg's linearization theorem for a stable hyperbolic equilibrium \cite{sternberg1957local}, the available NHIM linearization theorems yield sufficient conditions for \ref{item:nhim-linearizable} that are not necessary in general.
Since Condition~\ref{item:smooth-2} implies that $\att$ decomposes into quasiperiodic invariant tori in a certain way, sufficient conditions for \ref{item:nhim-linearizable} closest to necessary for our purposes might be obtained via techniques from the literature on normally hyperbolic invariant quasiperiodic tori \cite{haro2006param,haro2016param} (cf. \cite[Thm~4.1]{haro2006param}, \cite[Sec.~9]{mezic2020spectrum}).
Similarly, useful sufficient conditions for \ref{item:nhim-es-red} might be obtained via techniques from the related literature on reducibility of linear flows on vector bundles over quasiperiodic tori \cite{johnson1981smoothness,puig2002reducibility} (cf. \cite[Thm~1]{johnson1981smoothness}).
However, Condition~\ref{item:smooth-3} can hold even if \ref{item:nhim-es-ck-triv}, \ref{item:nhim-linearizable}, \ref{item:nhim-es-red} do not.
What is the gap between Condition~\ref{item:smooth-3} and \ref{item:nhim-es-ck-triv}, \ref{item:nhim-linearizable}, \ref{item:nhim-es-red}?
\end{Rem}

\section{Proofs}\label{sec:proof}
This section contains the proofs of Theorems~\ref{th:smooth-compact}, \ref{th:cont-compact}, \ref{th:cont}, \ref{th:smooth} and Propositions~\ref{prop:hopf}, \ref{prop:recog}.
The following lemma is Theorems~\ref{th:smooth-compact}, \ref{th:cont-compact} combined.
\begin{Lem}\label{lem:compact-case}
Fix $k\in \N_{\geq 1}\cup\{\infty\}$.
Let $\Phi$ be a $C^k$ (resp. $C^0$) flow on a compact $C^k$ manifold (resp. compact topological space homeomorphic to a subspace of a manifold) $\st$.
Then $(X,\Phi)$ is linearizable by a $C^k$ ($C^0$) embedding if and only if $\Phi$ is  a $1$-parameter subgroup of a $C^k$ ($C^0$ with finitely many orbit types) torus action.
\end{Lem}

\begin{proof}
First assume that $\Phi$ is  a $1$-parameter subgroup of a $C^k$ ($C^0$ with finitely many orbit types) action $\Theta\colon H\times X\to X$  of a torus $H$ on $X$.
By the $C^k$ version \cite[Thm~4.6.6.]{duistermaat2000lie} ($C^0$ version \cite[Thm~6.1]{mostow1957equivariant}) of the Mostow-Palais equivariant embedding theorem \cite{mostow1957equivariant,palais1957imbedding,palais1960classification}, there exists $n\in \N$, a $C^k$ ($C^0$) 
embedding $F\colon X\to \R^n$, and a $C^\infty$ Lie group homomorphism $\rho\colon H\to \text{GL}(\R^n)$ such that 
\begin{equation}\label{eq:equiv-emb}
F \circ \Theta^{h} = \rho(h)\circ F
\end{equation}
for all $h\in H$.
Since by assumption there is $\boldsymbol \omega$ in the Lie algebra of $H$ such that $\Phi^t = \Theta^{\exp(\boldsymbol \omega t)}$ for all $t\in \R$, \eqref{eq:equiv-emb} implies that
$$F\circ \Phi^t = F\circ \Theta^{\exp(\boldsymbol \omega t)}= \rho(\exp(\boldsymbol \omega t))\circ F = e^{(\rho_* \boldsymbol \omega)t}\circ F$$
for all $t\in \R$, where $\rho_*$ is the $\rho$-induced Lie algebra homomorphism \cite[Def.~1.10.2]{duistermaat2000lie} and the third equality follows from $\rho$ being a Lie group homomorphism \cite[Lem.~1.5.1]{duistermaat2000lie}.
Thus, $F$ is a $C^k$ ($C^0$) embedding of $(X,\Phi)$ into the linear flow generated by the linear vector field $B\coloneqq \rho_* \boldsymbol \omega$.

Next, assume that $(\st,\Phi)$ is linearizable by a  $C^k$ ($C^0$) embedding. 
Since $\st$ is compact, the Jordan normal form theorem implies that any linearizing $C^0$ embedding of $(\st,\Phi)$ sends $\st$ into the sum of real invariant linear subspaces corresponding to eigenspaces of purely imaginary eigenvalues of the matrix $B$ generating the linear flow.
Thus, after embedding the real linear space into a complex linear space via complexifying and diagonalizing $B$, we may assume that $\st\subset \C^n$ is an invariant subset for the linear flow generated by the ODE
\begin{equation}\label{eq:linear-ode}
\begin{split}
\dot{z}&= 2\pi i\,\text{diag}(\boldsymbol\omega)z
\end{split}
\end{equation}
on $\C^n$, where $i=\sqrt{-1}$, $z=(z_1,\ldots, z_n)\in \C^n$, $\boldsymbol \omega = (\omega_1,\ldots,\omega_n)\in \R^n$, and $\text{diag}(\boldsymbol\omega)$ is the diagonal matrix with $j$-th diagonal entry $\omega_j$.
Define a $C^\infty$ action $\Theta$ of $T^n=\R^n/\Z^n$ on $\C^n$ by $\Theta^\tau(z)\coloneqq \exp(2\pi i \text{diag}( \tau))z$.
Note that $\st$ is a union of closures of trajectories of \eqref{eq:linear-ode} since $\st$ is a closed subset of $\C^n$, with each such closure coinciding with an orbit $\Theta^H( z)$ of the  closure $H\subset T^n$ of the $1$-parameter subgroup $\boldsymbol{\omega}\R \mod 1$.
Thus, $\st$ is a union of $H$-orbits, so the $C^\infty$ $H$ action on $\C^n$ restricts to a well-defined $C^k$ ($C^0$ with finitely many orbit types \cite[Thm~1.8.4]{palais1960classification}) $H$ action on $\st$ with $\Phi$  a $1$-parameter subgroup of the restricted $H$ action.
Finally, the closed subgroup theorem \cite[Cor.~1.10.7]{duistermaat2000lie} and the classification of abelian Lie groups \cite[Cor.~1.12.4]{duistermaat2000lie} imply that $H$ is a $C^\infty$ Lie subgroup of $T^n$ isomorphic to a torus. 
This completes the proof.
\end{proof}

This completes the proofs of Theorems~\ref{th:smooth-compact}, \ref{th:cont-compact}. 
The remaining results are restated for convenience, then proved.
\PropHopf*

\begin{proof}
By Theorem~\ref{th:smooth-compact}, $\Phi$ is  a $1$-parameter subgroup of a $C^1$ torus action on $X$.
By restricting the torus action to the action of the closure $H$ of this $1$-parameter subgroup, we obtain a $C^1$ action $\Theta$ of a subtorus $H$ on $\st$ with the property that the equilibria of $\Phi$ coincide with the fixed points of $\Theta$.
Let $x\in \st$ be an isolated such equilibrium/fixed point.
By Bochner's linearization theorem, there is a $\Theta$-invariant open neighborhood $U$ of $x$ and $C^1$ local coordinates on $U$ in which $\Theta$ is the linear action of a closed subgroup of the orthogonal group $O(n)$ on $\R^n$ (\cite[Thm~2.2.1]{duistermaat2000lie}, \cite[Thm~4.7.1]{hirsch1976difftop}).
Thus, in these coordinates $\Phi$ corresponds to the action of a $1$-parameter subgroup of $O(n)$, and hence the vector field generating  $\Phi$ corresponds to a skew-symmetric linear vector field $B$.
Moreover, $B$ is invertible since $x$ is an isolated equilibrium of $\Phi$ and hence an isolated zero of $B$.
The fact that all invertible skew-symmetric matrices have even dimensions implies that $\dim \st$ is even.
Finally, the determinant of $B$ is positive since eigenvalues of an invertible real skew-symmetric matrix come in imaginary conjugate pairs, so the Hopf index of the vector field generating $\Phi$ at $x$ is equal to $1$.
\end{proof}

\PropRecog*

\begin{proof}
The condition $F\circ \Phi^t = \boldsymbol \omega t + F \mod 1$ implies that the open set $R\subset T^n$ of regular values of $F$ is invariant under the irrational linear flow $(t,y)\mapsto \boldsymbol{\omega}t+y\mod 1$.
Since an irrational linear flow is minimal, either $R = T^n$ or $R=\varnothing$.
Sard's theorem \cite[Thm~3.1.3]{hirsch1976difftop} eliminates the latter option, so $F$ is a local diffeomorphism, hence also a covering map since $\st$ is compact \cite[p.~303, 11-9]{lee2011topological}. 
The classification of covering spaces of $T^n$ \cite[\S 1.3]{hatcher2002algebraic} yields a natural $C^k$ identification of $\st$ with $T^n=\R^n/\Z^n$ such that $F\colon \R^n/\Z^n\approx \st\to \R^n/\Z^n$  is additionally a Lie group homomorphism, so is given by an invertible matrix $B$ with integer entries.
Since $F\circ \Phi^t = \boldsymbol \omega t + F \mod 1$, this implies that $\Phi$ is identified with the flow of the constant vector field $\tilde{\boldsymbol{\omega}}= B^{-1}\boldsymbol{\omega}$.
The entries of $B^{-1}$ are rational, so rational independence of the components of $\boldsymbol{\omega}$ implies the same for $\tilde{\boldsymbol{\omega}}$.
\end{proof}

\ThmCont*
\begin{proof}
Since $\st$ is locally compact and $\att$ is globally asymptotically stable, there exists a proper $C^0$ Lyapunov function $V\colon \st\to [0,\infty)$ satisfying $V^{-1}(0)=\att$ and strictly decreasing along $\Phi$-trajectories outside $\att$ \cite[Thm~2.7.20]{bhatia1967stability}.
Fix any $c\in (0,\infty)$ and define the compact set $N\coloneqq V^{-1}(c)$.
Each trajectory in $\st\setminus \att$ crosses $N$ exactly once, and the ``time-to-impact-$N$'' map $\tau\colon \st\setminus \att\to (-\infty,\infty)$ is $C^0$ (by the method of proof of \cite[Thm~2.7.14]{bhatia1967stability} or \cite[Theorem~II.2.3]{conley1978isolated}).	

We now show that any linearizing injective $C^0$ map $F\colon \st\to \R^n$ is proper.
Let $B\in \R^{n\times n}$ satisfy $F\circ \Phi^t = e^{Bt}\circ F$ for all $t\in \R$.
Since $F(N)$ and $F(\att)$ are disjoint compact sets satisfying $e^{Bt}F(N)\to F(\att)$ as $t\to\infty$, the Jordan normal form theorem implies that $e^{-Bt}F(N)\to \infty$ as $t\to\infty$.
Since $\tau(x)\to\infty$ as $V(x)\to\infty$, 
$$F(x) = e^{-B\tau(x)}F(\Phi^{\tau(x)}(x))\in e^{-B\tau(x)}F(N)\to \infty$$
for $x\in \st\setminus \att$ as $V(x)\to\infty$.
Since $V$ is proper, this shows that $F$ is proper.

Next, assuming that Conditions \ref{item:cont-1} and \ref{item:cont-2} hold, we show that $(\st,\Phi)$ is linearizable by a $C^0$ embedding.
By Condition~\ref{item:cont-2} and Theorem~\ref{th:cont-compact}, there exists  $n_0\in \N$, $B_0\in \R^{n_0\times n_0}$, and 
a $C^0$ embedding $F_0\colon \att\to \R^{n_0}$ that linearizes $(\att,\Phi|_{\R\times \att})$.
Since $\st$ is homeomorphic to a subspace of a manifold, the Whitney embedding theorem implies the existence of a $C^0$ embedding $F_1\colon N\to \R^{n_1}$ with image $F_1(N)$ contained in the unit sphere $S^{n_1-1}\subset \R^{n_1} $ \cite[p.~316]{munkres2000topology}.
Thus, the map $F\colon \st \to \R^{n_0}\times \R^{n_1}$ defined by
\begin{equation*}
	F(x)\coloneqq \begin{cases}
		(F_0(x),0), & x\in \att\\
		(F_0\circ P(x),e^{\tau(x)}F_1(\Phi^{\tau(x)}(x)), &x\not \in \att
	\end{cases}
\end{equation*}
is $C^0$ since $F_0\circ P|_A = F_0$ and since $\tau(x)\to -\infty$ as $x\to \att$, so $$e^{\tau(x)}F_1(\Phi^{\tau(x)}(x))\in e^{\tau(x)}S^{n_1-1}\to 0$$ as $x\to\att$. 
Moreover, $F$ is injective since $F(x)=F(y)$ implies that either (i) $x=y\in \att$, or (ii) $x$ and $y$ belong to the same trajectory and have the same impact time $\tau(x)=\tau(y)$. 
But if $x,y\in \st\setminus \att$ belong to the same trajectory and $\tau(x)=\tau(y)$, then $x = y$ since $\tau$ is injective on each trajectory.
This establishes that $F\colon \st\to \R^{n_0}\times \R^{n_1}$ is an injective $C^0$ map.
Since $\tau \circ \Phi^t|_{\st\setminus \att} = \tau-t$, 
\begin{align*}
	F\circ \Phi^t(x) &= (F_0\circ P\circ \Phi^t(x),e^{\tau\circ \Phi^t(x)}F_1(\Phi^{\tau\circ \Phi^t(x)}\circ\Phi^t(x)))\\
	&= (F_0\circ P \circ \Phi^t(x),e^{\tau(x)-t}F_1(\Phi^{\tau(x)-t}\circ \Phi^t(x))) \\&= (F_0\circ \Phi^t \circ P(x), e^{\tau(x)-t}F_1(\Phi^{\tau(x)}(x)) )\\
	&= (e^{B_0 t}\circ F_0\circ P(x), e^{-t} e^{\tau(x)}F_1(\Phi^{\tau(x)}(x)) )\\
	&= \exp\left( \begin{bmatrix}
		B_0 & 0\\
		0 & -I
	\end{bmatrix}t\right)F(x)
\end{align*}
for all $x\in \st\setminus \att$ and $t\in \R$, and
\begin{align*}
	F\circ \Phi^t(x)&= (F_0\circ \Phi^t(x),0)\\
	&= (e^{B_0 t} \circ F_0, e^{-t} 0)\\
	&= \exp\left( \begin{bmatrix}
		B_0 & 0\\
		0 & -I
	\end{bmatrix}t\right)F(x)
\end{align*}
for all $x\in \att$ and $t\in \R$.
This shows that the injective $C^0$ map $F$ is linearizing.
We have shown that such an $F$ is proper, hence a $C^0$ embedding \cite[Cor.~4.97(b)]{lee2011topological}.

Finally, assuming that $(X,\Phi)$ is linearizable by a (necessarily proper) $C^0$ embedding, we show that Conditions~\ref{item:cont-1} and \ref{item:cont-2} hold.
Restricting any linearizing $C^0$ embedding to $\att$ and invoking Lemma~\ref{lem:compact-case} immediately yields Condition \ref{item:cont-2}, so we need only verify Condition~\ref{item:cont-1}.

To do so, we may assume there is $n\in \N$ and $B\in \R^{n\times n}$ such that $\Phi$ is the restriction of the linear flow generated by $B$ to a closed invariant subset $\st\subset \R^n$.
Since  $\att$ is a globally asymptotically stable compact invariant set, $\st$ must be contained in the sum of the real invariant linear subspace $E_-\subset \R^n$ corresponding to generalized eigenspaces of eigenvalues with negative real part for $B$ and  the real invariant linear subspace $E_0$ corresponding to eigenspaces of purely imaginary eigenvalues.
Moreover, $A\subset E_0$.
By restricting $B$ to this sum, we may and do assume that  $\R^n = E_0 \oplus E_-$.
Let
$$P_0\colon \R^n = E_0 \oplus E_- \to E_0$$
be the linear projection to $E_0$ with kernel $E_-$.
Fix any $v_0\in E_0$ and $v_-\in E_-$ such that $v_0+v_-\in \st$.
As in the proof of Lemma~\ref{lem:compact-case}, $v_0=P_0(v_0+v_-)$ is contained in an invariant  $C^\infty$ embedded torus $T\subset E_0$ densely filled by each trajectory  within it.
Since $e^{Bt}v_-\to 0$ as $t\to\infty$, linearity implies that $T$ is equal to the omega-limit set of the trajectory $t\mapsto e^{Bt}(v_0+v_-)$.
Since this omega-limit set is necessarily contained in the $\Phi$-globally asymptotically stable set $\att$, this implies that $T \subset \att$.
Hence $P_0(v_0+v_-) = v_0\in T\subset \att$, so arbitrariness of $v_0+v_-\in \st$ implies that $P_0(\st)\subset \att$.
Since also $P_0|_\att = \id_\att$, the map $P\coloneqq P_0|_\st\colon \st\to \att$ is a well-defined $C^0$ retraction.
Moreover, $P_0$ commutes with $B$ by construction and hence $$P\circ \Phi^t = (P_0\circ e^{Bt})|_\st = (e^{Bt}\circ P_0)|_\st = \Phi^t\circ P$$
for all $t$.
Thus, Condition \ref{item:cont-1} holds.
\end{proof}

The proof that Conditions~\ref{item:smooth-1}, \ref{item:smooth-2}, \ref{item:smooth-3} in the following theorem are sufficient for linearizability by a $C^k$ embedding works by first constructing a linearizing $C^k$ embedding on a small neighborhood of $\att$, then using a globalization technique of Lan and Mezi\'{c} \cite{lan2013linearization} (see also \cite{mezic_book,eldering2018global,kvalheim2021existence}) to construct a linearizing $C^k$ embedding defined on all of $\st$.

\ThmSmooth*

\begin{proof}
First assume that $(\st,\Phi)$ is linearizable by a $C^k$ embedding.
That any such embedding is necessarily a proper map is immediate from Theorem~\ref{th:cont}.
We may therefore assume that $\st$ is a properly embedded $C^k$ submanifold of $\R^n$ for some $n\in \N$ and that $\Phi$ is the restriction of the linear flow generated by some $B\in \R^{n\times n}$.

As in the proof of Theorem~\ref{th:cont}, we may and do assume that all eigenvalues of $B$ have nonpositive real part.
Let $E_0$, $E_-$, and
$$P_0\colon \R^n = E_0\oplus E_-\to E_0$$ be as in the proof of Theorem~\ref{th:cont}.
The same argument from that proof implies that $P_0(\st)=\att$ and $P_0|_\att = \id_\att$, so that $P_0|_\st$ is a continuous retraction when viewed as a map into $\att$.
But $P_0|_\st$ is also $C^k$ when viewed as a map into $\st$, which implies that $\att$ is a $C^k$ embedded submanifold of $\st$ \cite[p.~20]{hirsch1976difftop}. 
Since $P_0$ commutes with the linear flow and hence $P_0|_\st$ commutes with $\Phi$, Condition~\ref{item:smooth-1} holds with $C^k$ asymptotic phase map $P\coloneqq P_0|_\st\colon \st\to \att$.
Condition~\ref{item:smooth-1} and Theorem~\ref{th:smooth-compact} in turn imply that Condition~\ref{item:smooth-2} holds.

To show that Condition~\ref{item:smooth-3} holds, let 
$$P_-\colon \R^n= E_0\oplus E_-\to E_-$$ be the linear projection with kernel $E_0$, and define $G\coloneqq P_-|_{\st}\colon \st\to E_{-}\approx \R^\ell$ with $\ell=\dim E_{-} $.
Since
$\R^n = \ker(P_-)\oplus\ker(P_0)$ and  $P|_\att = \id_\att$, we have $\T_\att \st = \ker(\T_\att G) \oplus \ker(\T_\att P)$ and  $\T_\att \st = \T \att \oplus \ker(\T_\att P)$.
Since $G|_\att = 0$ implies that  $\T \att \subset \ker(T_\att G)$, it follows that $\ker(\T_\att G) = \T \att$.  
And since $P_-$ commutes with $B$,
$$G\circ \Phi^t = (P_-\circ e^{Bt})|_{\st}=(e^{Bt}\circ P_-)|_\st = e^{Bt}\circ G$$
for all $t\in \R$.  
Thus, the final Condition~\ref{item:smooth-3} holds.

 Next assume that Conditions \ref{item:smooth-1}, \ref{item:smooth-2}, \ref{item:smooth-3} hold.
Since the vector field generating $\Phi$ is uniquely integrable, there exists a proper strict $C^k$ Lyapunov function $V\colon \st\to [0,\infty)$ for $\att= V^{-1}(0)$ with respect to $\Phi$ (\cite[Thm~3.2]{wilson1969smooth}, \cite[\S 6]{fathi2019smoothing}, Remark~\ref{rem:c1-vs-cinf}).
Since $V$ is proper, there is $c\in (0,\infty)$ such that $V^{-1}([0,c])\subset U$.
A standard implicit function theorem argument implies that $N\coloneqq V^{-1}(c)\subset U$ is a $C^k$ embedded submanifold and the ``time-to-impact-$N$'' map $\tau\colon \st\setminus \att\to (-\infty,\infty)$ is $C^k$.
Consider $F_0\colon \st\to \R^\ell$ given by 
\begin{equation*}
F_0(x)\coloneqq \begin{cases}
G(x), & x\in U\\
e^{-B\tau(x)}G(\Phi^{\tau(x)}(x)), & x\in \st\setminus \att
\end{cases}
\end{equation*}
and note that $F_0$ is well-defined, since if $x\in U\setminus \att $ then   
$$e^{-B\tau(x)}G(\Phi^{\tau(x)}(x)) = e^{-B\tau(x)}e^{B\tau(x)}G(x)=G(x) $$
by Condition~\ref{item:smooth-3}.
Since $U$ and $\st\setminus \att$ are open, $F_0\colon \st\to \R^\ell$ is $C^k$.

By Condition~\ref{item:smooth-2} and Theorem~\ref{th:smooth-compact}, there exists  $\ell_1\in \N$, $B_1\in \R^{\ell_1\times \ell_1}$, and 
a $C^k$ embedding $F_1\colon \att\to \R^{\ell_1}$ that linearizes $(\att,\Phi|_{\R\times \att})$.
Define the $C^k$ map $F\colon \st\to \R^{\ell_1}\times \R^\ell$ by $F(x)\coloneqq (F_1\circ P(x), F_0(x)).$
Since
$$\ker\T F= \ker \T(F_1\circ P)\cap \ker \T F_0= \ker \T P \cap \ker \T F_0 $$
and the intersection of the rightmost term with $\T_\att \st$ is equal to $\ker T_\att P \cap \T A = 0_{\T_\att \st}$ by Condition~\ref{item:smooth-3}, continuity implies that there is an open set $W_0\supset \att$ such that $\ker(\T_{W_0}F)=0_{\T_{W_0}\st}$ and hence $F|_{W_0}$ is an immersion.

Since $F|_{W_0}$ is an immersion and $F|_{\att}$ is injective, there is an open set $W\supset A$ contained in $W_0$ such that the restriction $F|_W$ is injective.
To show that $F$ itself is injective, fix $x,y\in \st$ such that $F(x)=F(y)$.
We will show that $x=y$ using injectivity of $F|_W$ and the fact that 
\begin{equation}\label{eq:F-conj-pf}
F\circ \Phi^t = \exp\left(\begin{bmatrix}
B_1 & 0\\
0 & B
\end{bmatrix}t \right) \circ F
\end{equation}
for all $t\in \R$, which holds since $$F_1\circ P\circ \Phi^t = F_1\circ \Phi^t\circ P = e^{B_1 t}\circ F_1\circ P,$$
$$F_0\circ \Phi^t|_\att = G\circ \Phi^t|_\att = e^{Bt}\circ G|_\att = e^{Bt}\circ F_0|_\att,$$
and, for all $x\in \st \setminus \att$,
\begin{align*}
F_0\circ \Phi^t(x)&= e^{-B\tau\circ \Phi^t(x)}G \circ \Phi^{\tau \circ \Phi^t(x)}\circ \Phi^t(x)\\
&= e^{-B(\tau(x)-t)}\circ G\circ \Phi^{\tau(x)-t}\circ \Phi^t(x)	\\
&= e^{Bt}e^{-B\tau(x)}G(\Phi^{\tau(x)}(x))\\
&= e^{Bt}F_0(x),
\end{align*}
where we have used that $\tau\circ \Phi^t|_{\st\setminus \att}=\tau - t$.
Global asymptotic stability of $\att$ implies the existence of $s>0$ such that $\Phi^s(x),\Phi^s(y)\in W$.
Since $F(x)=F(y)$ and \eqref{eq:F-conj-pf} imply that $F|_W( \Phi^s(x)) = F|_W(\Phi^s(y))$, injectivity of $F|_W$ and $\Phi^s$ implies that $x=y$.
Thus, $F$ is injective.
Similarly, taking tangent maps of both sides of \eqref{eq:F-conj-pf} reveals that $F$ is an immersion.
Since $F$ is also a proper map by Theorem~\ref{th:cont}, $F$ is a $C^k$ embedding \cite[Cor.~4.97]{lee2011topological} that linearizes $(\st,\Phi)$ by \eqref{eq:F-conj-pf}.
\end{proof}

\section{Conclusion}\label{sec:conclusion}
We obtained necessary and sufficient conditions for linearizability of $(\st,\Phi)$ by a $C^{k\geq 0}$ embedding when $\st$ is either compact or  $\Phi$ has a compact global attractor, where $\st$ is a $C^k$ manifold if $k\in \N_{\geq 1}\cup\{\infty\}$ and  $\st$ is a ``reasonable'' topological space if $k=0$ (Theorems~\ref{th:smooth-compact}, \ref{th:cont-compact}, \ref{th:cont}, \ref{th:smooth}).
We also obtained a separate sufficient condition for $C^{k\geq 1}$ linearizability in the compact case (Proposition~\ref{prop:recog}), a necessary condition for the same (Proposition~\ref{prop:hopf}) having implications for linearizability in the presence of isolated equilibria (Corollary~\ref{co:odd-isolated-no-embed}, \ref{co:euler}, \ref{co:surfaces}), a necessary condition for $C^0$ linearizability (Corollary~\ref{co:motivate}), and extensions of the classical Hartman-Grobman and Floquet normal form theorems (Corollaries~\ref{co:hartman-grobman}, \ref{co:floquet}).
Additionally, we illustrated the theory in several examples (Examples~\ref{ex:zero-vf}, \ref{ex:basic-tori}, \ref{ex:sphere}, \ref{ex:klein}, \ref{ex:rp2}, \ref{ex:stabilize-earlier}, \ref{ex:nonhyp-po-nophase}, \ref{ex:nonresonance}). 

In particular, our results completely characterize linearizability by $C^{k\geq 0}$ embeddings for flows on connected state spaces that are either compact or contain at least one nonempty compact attractor (Theorems~\ref{th:smooth-compact}, \ref{th:cont-compact}, \ref{th:cont}, \ref{th:smooth}, Corollary~\ref{co:motivate}).
While our results still furnish \emph{necessary} conditions for the remaining case of a noncompact state space without any nonempty compact attractors, it would be interesting to fully characterize this remaining case as well.
Additionally, even in the $C^{k\geq 1}$ global attractor case covered by Theorem~\ref{th:smooth}, the relationship between one of our hypotheses and related literature remains to be fully understood (Remark~\ref{rem:about-cond-3}).

Finally, we note that our results have implications for the field of ``applied Koopman operator theory'', where algorithms like ``extended Dynamic Mode Decomposition'' (\cite{lusch2018deep}, \cite[\S 5.1, 5.4]{brunton2022modern}, \cite{haller2024data}) are used to attempt to numerically compute linearizing embeddings.
Our results give precise conditions under which linearizing embeddings exist, imposing fundamental limitations on such algorithms complementary to practical limitations identified by practitioners  \cite{wu2021challenges,haller2024data}.

\section*{Acknowledgments}
This material is based upon work supported by the Air Force Office of Scientific Research under award number FA9550-24-1-0299, managed by Dr. Frederick A. Leve.
We thank Robert Cardona for communicating to us a generalized version of Tao's universality notion as in \S\ref{sec:intro}, Zexiang Liu and Eduardo Sontag for useful discussions about the relation of our results with those of \cite{liu2023non,liu2025properties}, and  Paulo Tabuada for references to the control theory and polynomial flows literatures.
We are also extremely grateful to Zexiang Liu for alerting us to an error in a previous version of Theorem~\ref{th:cont}.

\bibliographystyle{unsrt}
\bibliography{ref}

%\clearpage    
%\input{author-response.tex}    

\end{document}